\documentclass{article} 
\usepackage[utf8]{inputenc}
\usepackage{amsmath,amsthm,amssymb,amsfonts}
\usepackage{verbatim}
\usepackage{array}

\usepackage{hyperref}

\usepackage{stmaryrd} 
\usepackage{hyperref}

\usepackage[colorinlistoftodos]{todonotes}

\theoremstyle{plain}
\newtheorem{thm}{Theorem}
\newtheorem{lem}[thm]{Lemma}
\newtheorem{prop}[thm]{Proposition}
\newtheorem{definition}[thm]{Definition}

\newtheorem{remark}[thm]{Remark}

\theoremstyle{definition}

\newtheorem{exl}[thm]{Example}

\numberwithin{thm}{section}

\newcommand{\adj}{\leftrightarrow}
\newcommand{\adjeq}{\leftrightarroweq}

\DeclareMathOperator{\id}{id}

\DeclareMathOperator{\Fix}{Fix}
\newcommand{\Z}{\mathbb{Z}}

\title{Subsets and Freezing Sets in the Digital Plane}
\author{Laurence Boxer
\thanks{Department of Computer and Information
Sciences, Niagara University, NY 14109, USA; and
Department of Computer Science and Engineering,
State University of New York at Buffalo.
Email: boxer@niagara.edu}
}
\date{}

\begin{document}

\maketitle

\begin{abstract}
    We continue the study of freezing sets for digital images introduced
    in~\cite{bs19a,BxFpSets,BxConvex}. We prove methods for obtaining
    freezing sets for digital images $(X,c_i)$ for $X \subset \Z^2$ and
    $i \in \{1,2\}$. We give examples to show how these methods can lead
    to the determination of {\em minimal} freezing sets.
    

\end{abstract}

\section{Introduction}

A digital image is a graph typically used to
model an object in Euclidean space that it
represents. Researchers in digital topology have 
had much success using methods inspired by 
classical topology to show that
digital images have properties such as
connectedness, continuous function, homotopy, 
fundamental group, homology, automorphism group, 
Euler characteristic, et al., analogous to those
of the objects represented.

However, the fixed point properties of a 
Euclidean object and its digital representative 
are often quite different. If $f: X \to X$ is
a continuous function on a Euclidean space,
knowledge of the fixed point set of $f$,
$\Fix(f)$, often tells us little about
$f|_{X \setminus \Fix(f)}$. By contrast,
if $f: (X,\kappa) \to (X,\kappa)$ is a
digitally continuous function, knowledge of 
$\Fix(f)$ often tells us 
much~\cite{bs19a,BxFpSets,BxConvex} about
$f|_{X \setminus \Fix(f)}$. 

The study of {\em freezing 
sets}~\cite{BxFpSets,BxConvex} helps us
deal with the following question:
If $f: (X,\kappa) \to (X,\kappa)$ is a
digitally continuous function and
$A \subset \Fix(f)$, must $f = \id_X$?
In this paper, we expand our knowledge of
freezing sets in digital images.

\section{Preliminaries}
Much of this section is quoted or paraphrased
from~\cite{BxFpSets,BxConvex} and other
references.

We use $\Z$ to indicate the set of integers.

\subsection{Adjacencies}
The $c_u$-adjacencies are commonly used 
in digital topology.
Let $x,y \in \Z^n$, $x \neq y$, where we consider these points as $n$-tuples of integers:
\[ x=(x_1,\ldots, x_n),~~~y=(y_1,\ldots,y_n).
\]
Let $u \in \Z$,
$1 \leq u \leq n$. We say $x$ and $y$ are 
{\em $c_u$-adjacent} if
\begin{itemize}
\item there are at most $u$ indices $i$ for which 
      $|x_i - y_i| = 1$, and
\item for all indices $j$ such that $|x_j - y_j| \neq 1$ we
      have $x_j=y_j$.
\end{itemize}
Often, a $c_u$-adjacency is denoted by the number of points
adjacent to a given point in $\Z^n$ using this adjacency.
E.g.,
\begin{itemize}
\item In $\Z^1$, $c_1$-adjacency is 2-adjacency.
\item In $\Z^2$, $c_1$-adjacency is 4-adjacency and
      $c_2$-adjacency is 8-adjacency.
\item In $\Z^3$, $c_1$-adjacency is 6-adjacency,
      $c_2$-adjacency is 18-adjacency, and $c_3$-adjacency
      is 26-adjacency.
\end{itemize}

For $\kappa$-adjacent $x,y$, we write $x \adj_{\kappa} y$ or $x \adj y$ when $\kappa$ is understood.
We write $x \adjeq_{\kappa} y$ or $x \adjeq y$ to mean that either $x \adj_{\kappa} y$ or $x = y$.

We say $\{x_n\}_{n=0}^k \subset (X,\kappa)$ is a {\em $\kappa$-path} (or a {\em path} if $\kappa$ is understood)
from $x_0$ to $x_k$ if $x_i \adjeq_{\kappa} x_{i+1}$ for $i \in \{0,\ldots,k-1\}$, and $k$ is the {\em length} of the path.

A subset $Y$ of a digital image $(X,\kappa)$ is
{\em $\kappa$-connected}~\cite{Rosenfeld},
or {\em connected} when $\kappa$
is understood, if for every pair of points $a,b \in Y$ there
exists a $\kappa$-path in $Y$ from $a$ to $b$.

We define
\[ N(X,\kappa, x) = \{ y \in X \, | \, x \adj_{\kappa} y\}.
\]
\[ N^*(X,\kappa, x) = \{ y \in X \, | \, x \adjeq_{\kappa} y\} =
   N(X,\kappa, x) \cup \{x\}.
\]

\begin{definition}
{\rm \cite{BxConvex}}
\label{bdDef}
Let $X \subset \Z^n$.
The {\em boundary of $X$ with respect to the $c_i$ adjacency},
$i \in \{1,2\}$, is
\[Bd_i(X) = \{x \in X \, | \mbox{ there exists } y \in \Z^n \setminus X \mbox{ such that } y \adj_{c_i} x\}.
\]
\end{definition}
Note $Bd_1(X)$ is what is called the {\em boundary of $X$}
in~\cite{RosenfeldMAA}. However, for this paper, $Bd_2(X)$ offers
certain advantages.

\subsection{Digitally continuous functions}
Material in this section is quoted or paraphrased
from~\cite{BxFpSets}.

The following generalizes a definition of~\cite{Rosenfeld}.

\begin{definition}
\label{continuous}
{\rm ~\cite{Bx99}}
Let $(X,\kappa)$ and $(Y,\lambda)$ be digital images. 
A function $f: X \rightarrow Y$ is 
{\em $(\kappa,\lambda)$-continuous} if for
every $\kappa$-connected $A \subset X$ we have that
$f(A)$ is a $\lambda$-connected subset of $Y$.
If $(X,\kappa)=(Y,\lambda)$, we say such a function is {\em $\kappa$-continuous},
denoted $f \in C(X,\kappa)$.
$\Box$
\end{definition}

When the adjacency relations are understood, we may simply say that 
$f$ is \emph{continuous}. Continuity can be expressed in terms of 
adjacency of points:
\begin{thm}
{\rm ~\cite{Rosenfeld,Bx99}}
A function $f: (X,\kappa) \to (Y,\lambda)$ is continuous if and only if $x \adj_{\kappa} x'$ in $X$ implies 
$f(x) \adjeq_{\lambda} f(x')$.
\end{thm}

Similar notions are referred to as {\em immersions}, 
{\em gradually varied operators}, and {\em gradually varied mappings}
in~\cite{Chen94,Chen04}.

For a positive integer $n$ and $i \in \{1, \ldots, n\}$
let $p_i: \Z^n \to \Z$ be the $i^{th}$ {\em projection function} defined as follows.
For $x=(x_1,\ldots, x_n) \in \Z^n$, $p_i(X) = x_i$.

\subsection{Digital disks and bounding curves}
Material in this section is largely quoted or paraphrased 
from~\cite{BxConvex}.

A $c_2$-connected set $S=\{x_i\}_{i=1}^n \subset \Z^2$ is a
{\em (digital) line segment} if the members of $S$ are collinear.

\begin{remark}
\label{segSlope}
{\rm ~\cite{BxConvex}}
A digital line segment must be vertical, horizontal, or have
slope of $\pm 1$. We say a segment with slope of $\pm 1$ is
{\em slanted}.
\end{remark}

A {\em (digital) $\kappa$-closed curve} is a
path $S=\{s_i\}_{i=0}^{m-1}$ such that $s_0=s_{m-1}$,
and $0 < |i - j| < m-1$ 
implies $s_i \neq s_j$. If 
$s_i \adj_{\kappa} s_j$ implies 
$|i - j| \mod m = 1$, $S$ is a {\em (digital) 
$\kappa$-simple closed curve}.
For a simple closed curve $S \subset \Z^2$ we generally assume
\begin{itemize}
    \item $m \ge 8$ if $\kappa = c_1$, and
    \item $m \ge 4$ if $\kappa = c_2$.
\end{itemize}
These are necessary for the Jordan Curve
Theorem of digital topology, below, as a
$c_1$-simple closed curve in $\Z^2$ must have at least 8 points to
have a nonempty finite complementary $c_2$-component,
and a $c_2$-simple closed curve in $\Z^2$ must have at least 4 points to
have a nonempty finite complementary $c_1$-component.
Examples in~\cite{RosenfeldMAA} show why it is
desirable to consider $S$ and $\Z^2 \setminus S$
with different adjacencies.

\begin{thm}
{\rm \cite{RosenfeldMAA}}
{\em (Jordan Curve Theorem for digital topology)}
Let $\{\kappa, \kappa'\} = \{c_1, c_2\}$.
Let $S \subset \Z^2$ be a simple closed 
$\kappa$-curve such that $S$ has at least 8 points if
$\kappa = c_1$ and such that $S$ has at least 
4 points if $\kappa = c_2$. Then
$\Z^2 \setminus S$ has exactly 2 $\kappa'$-connected
components.
\end{thm}

One of the $\kappa'$-components of 
$\Z^2 \setminus S$ is finite and the other is infinite. This 
suggests the following.
\begin{definition}
\label{diskDef}
{\rm ~\cite{BxConvex}}
Let $S \subset \Z^2$ be a $c_2$-closed curve such that
$\Z^2 \setminus S$ has two $c_1$-components, one finite and the
other infinite. The union $D$ of $S$ and the finite $c_1$-component 
of $\Z^2 \setminus S$ is a {\em (digital) disk}. $S$ is
a {\em bounding curve} of $D$. The finite $c_1$-component 
of $\Z^2 \setminus S$ is the {\em interior of} $S$, denoted $Int(S)$,
and the infinite $c_1$-component of $\Z^2 \setminus S$ is the {\em exterior of} 
$S$, denoted $Ext(S)$.
\end{definition}

\begin{definition}
\label{thickness}
{\rm \cite{BxConvex}}
Let $X \subset \Z^2$ be a digital disk. We say $X$ is
{\em thick} if the following are satisfied. For some bounding
curve $S$ of $X$,
\begin{itemize}
    \item for every slanted segment~$S$ of $Bd_2(X)$,
 if  $p \in S$ is not an endpoint of  $S$, 
then there exists $c \in X$ such that 
(see Figure~\ref{fig:innerBdPt})
\begin{equation}
    \label{slantSegProp}
   c \adj_{c_2} p \not \adj_{c_1} c,
\end{equation}
and
\item if $p$ is the vertex of a 90$^\circ$ ($ \pi / 2$
      radians) interior angle $\theta$ of $S$, then there
      exists $q \in Int(X)$ such that
      \begin{itemize}
          \item if $\theta$ has horizontal and vertical sides then
                $q \adj_{c_2} p \not \adj_{c_1} q$ (see
                Figure~\ref{fig:degrees90a});
          \item if $\theta$ has slanted sides then
                $q \adj_{c_1} p$ (see
                Figure~\ref{fig:degrees90b});
      \end{itemize}
      and
\item if $p$ is the vertex of a 135$^\circ$ ($3 \pi / 4$
      radians) interior angle $\theta$ of $S$,
      there exist $b,b' \in X$
      such that $b$ and $b'$ are in the interior of $\theta$ and
      (see Figure~\ref{fig:degrees135c1})
      \[ b \adj_{c_2} p \not \adj_{c_1} b~~~ \mbox{ and }~~~ 
      b' \adj_{c_1} p.
       \]
\end{itemize}
\end{definition}

    \begin{figure}
        \centering
        \includegraphics[height=1in]{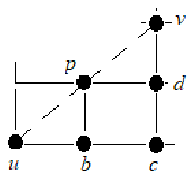}
        \caption{\cite{BxConvex}
        $p \in \overline{uv}$ in a bounding curve,
        with $\overline{uv}$ slanted.
        Note $u \not \adj_{c_1} p \not \adj_{c_1} v$,
        $p \adj_{c_2} c \not \adj_{c_1} p$,
        $\{p,c\} \subset N(\Z^2,c_1,b) \cap N(\Z^2,c_1,d)$. If
        $X$ is thick then $c \in X$.
        (Not meant to be understood as showing all of $X$.)}
        \label{fig:innerBdPt}
    \end{figure}
    
        \begin{figure}
        \centering
        \includegraphics[height=0.75in]{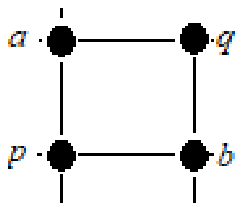}
        \caption{\cite{BxConvex} $\angle apb$ is a
        $90^{\circ}$ ($\pi/2$ radians)
        angle of a bounding curve of $X$ at $p \in A_1$, with
        horizontal and vertical sides. If $X$ is thick then
        $q \in Int(X)$. (Not meant to
        be understood as showing all of $X$.)}
        \label{fig:degrees90a}
    \end{figure}
    
       \begin{figure}
        \centering
        \includegraphics[height=1in]{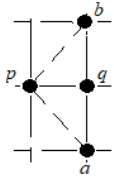}
        \caption{\cite{BxConvex} $\angle apb$ is a
        $90^\circ$ ($\pi/2$ radians) angle
         between slanted segments of a bounding curve. If $X$ is
         thick then $q \in Int(X)$. (Not meant to
        be understood as showing all of $X$).}
        \label{fig:degrees90b}
    \end{figure}

         \begin{figure}
        \centering
        \includegraphics[height=1in]{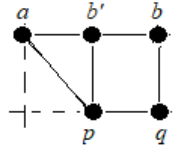}
        \caption{\cite{BxConvex} $\angle apq$ is an angle of
        135$^ \circ$ degrees ($3 \pi /4$ radians)
        of a bounding curve of $X$ at $p$, with
        $\overline{ap} \cup \overline{pq}$
            a subset of the bounding curve. If
            $X$ is thick then $b,b' \in X$. (Not meant to
        be understood as showing all of $X$.)
        }
        \label{fig:degrees135c1}
        \end{figure}

\subsection{Tools for determining fixed point sets}
Material in this section is largely quoted or
paraphrased from~\cite{BxConvex}
and other references as indicated.

The following assertions are useful in
determining fixed point and freezing sets.

\begin{prop}
\label{uniqueShortestProp}
{\rm (Corollary~8.4 of~\cite{bs19a})}
Let $(X,\kappa)$ be a digital image and
$f \in C(X,\kappa)$. Suppose
$x,x' \in \Fix(f)$ are such that
there is a unique shortest
$\kappa$-path $P$ in~$X$ from $x$ 
to $x'$. Then $P \subset \Fix(f)$.
\end{prop}

Lemma~\ref{cuPulling} below is in the spirit of ``pulling" as
introduced in~\cite{hmps}.
We quote~\cite{BxFpSets}:
\begin{quote}
    The following assertion can
be interpreted to say that
in a $c_u$-adjacency,
a continuous function that moves
a point~$p$ also [pulls along]
a point that is ``behind"
$p$. E.g., in $\Z^2$, if $q$ and $q'$ are
$c_1$- or $c_2$-adjacent with $q$
left, right, above, or below $q'$, and a
continuous function $f$ moves $q$ to the left,
right, higher, or lower, respectively, then
$f$ also moves $q'$ to the left,
right, higher, or lower, respectively.
\end{quote}

\begin{lem}
\label{cuPulling}
{\rm \cite{BxFpSets}}
Let $(X,c_u)\subset \Z^n$ be a digital image, 
$1 \le u \le n$. Let $q, q' \in X$ be such that
$q \adj_{c_u} q'$.
Let $f \in C(X,c_u)$.
\begin{enumerate}
    \item If $p_i(f(q)) > p_i(q) > p_i(q')$
          then $p_i(f(q')) > p_i(q')$.
    \item If $p_i(f(q)) < p_i(q) < p_i(q')$
          then $p_i(f(q')) < p_i(q')$.
\end{enumerate}
\end{lem}

\begin{figure}
    \centering
    \includegraphics[height=1in]{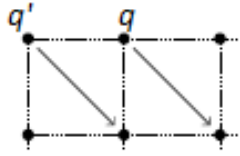}
    \caption{\cite{BxConvex}
    Illustration of Lemma~\ref{cuPulling}. Arrows show
    the images of $q,q'$ under $f \in C(X,c_2)$. Since
    $f(q)$ is to the right of $q$ and $q' \adj_{c_1,c_2} q$ with
    $q'$ to the left of $q$, $f$ pulls $q'$ to the right so that
    $f(q')$ is to the right of $q'$.
    }
    \label{fig:pull}
\end{figure}

Figure~\ref{fig:pull} illustrates Lemma~\ref{cuPulling}.

\begin{thm}
\label{bdCurveFreezeSet}
{\rm \cite{BxConvex}}
Let $D$ be a digital disk in $\Z^2$. Let
$S$ be a bounding curve for $D$. Then $S$ is
a freezing set for $(D,c_1)$ and for $(D,c_2)$.
\end{thm}

\begin{lem}
\label{c2shortSlant}
Let $X \subset \Z^2$ and let $a,b \in X$ be
such that $a$ and $b$ are endpoints of a
slanted digital line segment $P \subset X$.
Let $f \in C(X,c_2)$ such that 
$\{a,b\} \subset \Fix(f)$. 
Then $P \subset \Fix(f)$.
\end{lem}

\begin{proof}
This assertion was proven in the proof of Theorem~4.2 of~\cite{BxConvex}.
\end{proof}

We will use the following.
\begin{definition}
{\rm \cite{BxConvex}}
\label{closeNbrDef}
Let $(X,\kappa)$ be a digital image. Let
$p,q \in X$ such that 
\[ N(X,p,\kappa) \subset N^*(X,q,\kappa).
\]
Then $q$ is a {\em close $\kappa$-neighbor}
of $p$.
\end{definition}

We say $X \subset \Z^2$ is 
\begin{itemize}
    \item {\em symmetric with respect to the $x$-axis} if $(x,y) \in X$ implies $(x,-y) \in X$;
    \item {\em symmetric with respect to the $y$-axis} if $(x,y) \in X$ implies $(-x,y) \in X$;
    \item {\em symmetric with respect to the origin} if $(x,y) \in X$ implies $(-x,-y) \in X$.
\end{itemize}

\begin{prop}
Let $X$ be a digital image.
\begin{itemize}
    \item Suppose $X \subset \Z^2$ is symmetric with respect to the $x$-axis. If $p=(x,y) \in X$
          has a close $c_i$-neighbor in $X$, then $p'=(x,-y)$ has a close $c_i$-neighbor,
          $i \in \{1,2\}$.
    \item Suppose $X \subset \Z^2$ is symmetric with respect to the $y$-axis. If $p=(x,y) \in X$
          has a close $c_i$-neighbor in $X$, then $p'=(-x,y)$ has a close $c_i$-neighbor,
          $i \in \{1,2\}$.
    \item Suppose $X \subset \Z^n$ is symmetric with respect to the origin and $1 \le u \le n$.
          If $p = (x,y) \in X$ has a close $c_u$ neighbor in $X$, then $p'=(-x,-y)$
           has a close $c_u$ neighbor in $X$.
\end{itemize}
\end{prop}

\begin{proof}
These assertions follow easily from Definition~\ref{closeNbrDef}.
\end{proof}

Note these assertions are easily generalized to symmetry with respect to an arbitrary
horizontal line, vertical line, or point, respectively.

\begin{exl}
A point $p$ with a close $\kappa$-neighbor $q$ need not be $\kappa$-adjacent to $q$.
In the disk shown in Figure~\ref{fig:convexRlts}, $(1,1)$ is a close $c_1$-neighbor of $(0,0)$ but
$(0,0)$ and $(1,1)$ are not $c_1$-adjacent. In the $c_2$-curve
\[X = \{(1,0), (0,1), (-1,0), (0,-1)\},\]
$(-1,0)$ is a close $c_2$-neighbor of $(1,0)$, but
$(1,0)$ and $(-1,0)$ are not $c_2$-adjacent.
\end{exl}

\begin{lem}
{\rm \cite{bs19a,BxConvex}}
\label{closeNbr}
Let $(X,\kappa)$ be a digital image. Let
$p,q \in X$ such that $q$ is a close 
$\kappa$-neighbor of $p$. Then $p$ belongs to every
freezing set of $(X,\kappa)$.
\end{lem}

However, in general a point of a freezing set for $(X,\kappa)$ need not have a close 
$\kappa$-neighbor in~$X$, as shown by the following.

\begin{exl}
\label{3cubeExl}
Let $X = [0,1]_{\Z}^3$. Let
\[ A = \{(0,0,0), (0,1,1), (1,0,1), (1,1,0)\}.
\]
See Figure~\ref{fig:unit3cube}. 
Then $A$ is a minimal freezing set for $(X,c_1)$~\cite{BxFpSets}.
However, it is easily seen that no member of $A$ has a close $c_1$-neighbor in~$X$.
\end{exl}

\begin{figure}
    \centering
    \includegraphics{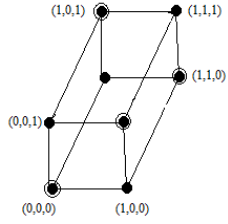}
    \caption{
    The unit 3-cube $X$, image of Example~\ref{3cubeExl}.
    Circled points make up a minimal $c_1$-freezing set, no member of which has
    a close $c_1$-neighbor in~$X$.
    }
    \label{fig:unit3cube}
\end{figure}

\section{$c_1$ results}
In this section, we obtain results for freezing sets $(X,c_1)$, with
$X \subset \Z^2$.

\begin{thm}
{\rm \cite{BxConvex}}
\label{convDiskThmActual}
Let $X$ be a thick convex disk with a
    bounding curve $S$.
    Let $A_1$ be the set of points $x \in S$ such that
$x$ is an endpoint of a maximal horizontal or a
maximal vertical edge of $S$. Let $A_2$ 
be the union of slanted line segments in $S$.
Then $A = A_1 \cup A_2$ is a minimal 
freezing set for $(X,c_1)$ (see  
Figure~\ref{fig:convexRlts}(ii)).
\end{thm}

\begin{figure}
    \centering
    \includegraphics{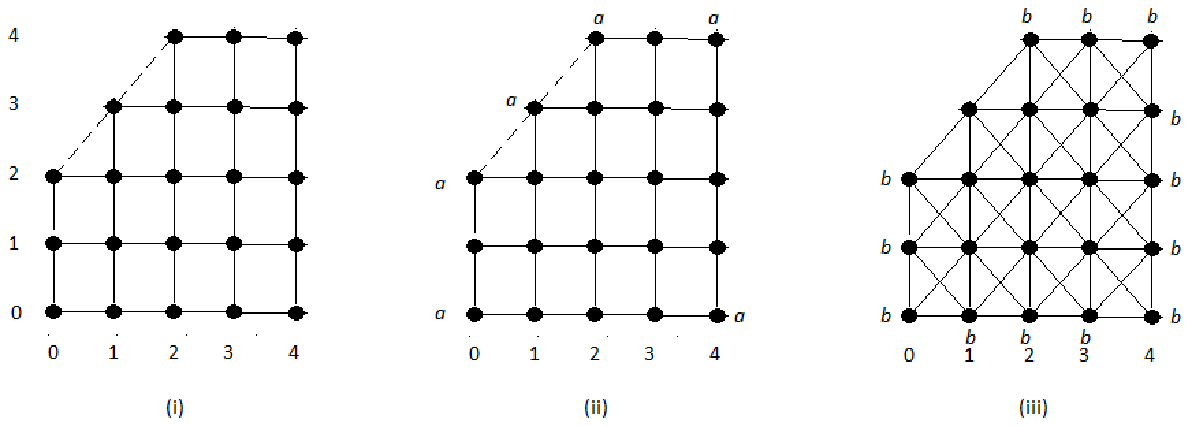}
    \caption{\cite{BxConvex} The convex disk 
    $D = [0,4]_{\Z}^2 \setminus \{(0,3),(0,4),(1,4)\}$. The dashed
    segment from $(0,2)$ to $(2,4)$ shown in (i) and (ii) indicates
    part of a bounding curve and not $c_1$-adjacencies. \newline
    (i) $D$ with a $c_2$ bounding curve. \newline
    (ii) $(D,c_1)$ with members of a minimal freezing set $A$
    marked ``{\em a}" - these are the endpoints of the maximal 
    horizontal and vertical segments of the bounding curve, 
    and all points of the slanted segment of the bounding curve,
    per Theorem~\ref{convDiskThmActual}. \newline
    (iii) $(D,c_2)$ with members of a minimal freezing set 
    $B$ marked ``{\em b}" - these are the endpoints of the maximal
    slanted edge and all the points of the horizontal and vertical
    edges of the bounding curve,
    per Theorem~\ref{convDiskThmC2Actual}.
    }
    \label{fig:convexRlts}
\end{figure}

\begin{thm}
\label{corners}
Let $V_i \subset X \subset \Z^2$, $i \in \{1,\ldots,n\}$ where
each $V_i$ is a thick convex disk. 
Let $X' = \bigcup_{i=1}^n V_i$.
Let $C_i$ be a bounding curve of $V_i$.
Let $A_{1,i}$ be the set of endpoints of maximal horizontal or vertical
segments of $C_i$.
Let $A_{2,i}$ be the union of maximal slanted
segments of $C_i$.
Then
$A = (X \setminus X') \cup \bigcup_{i=1}^n (A_{1,i} \cup A_{2,i})$ is a freezing set for $(X,c_1)$.
\end{thm}

\begin{proof}
Let $f \in C(X,c_1)$ such that 
$A \subset \Fix(f)$. For each~$i$, it follows from
Proposition~\ref{uniqueShortestProp} that
the horizontal and vertical segments whose endpoints are in~$A_{1,i}$
belong to $\Fix(f)$; and it follows from our choice of $A_{2,i}$ that
$C_i \subset \Fix(f)$. It follows from
Proposition~\ref{uniqueShortestProp} that
each horizontal segment joining two points of $C_i$
belongs to $\Fix(f)$. Since $V_i$ is convex, therefore
$V_i \subset \Fix(f)$; hence $X' \subset \Fix(f)$.
Since by hypothesis, 
$X \setminus X' \subset A \subset \Fix(f)$,
we must have $Fix(f) = X$, and the assertion
follows.
\end{proof}

In the following example, we show that the sets $\{V_i\}_{i=1}^n$ and $A$
of Theorem~\ref{corners} are not in general unique, and $A$ may not be minimal.
\begin{exl}
\label{unionRectanglesExl}
Let $X = ([0,2]_{\Z} \times [0,2]_{\Z}) \cup
          ([2,4]_{\Z} \times[0,3]_{\Z})$
(see Figure~\ref{fig:unionRectangles}),
for which the union above yields from Theorem~\ref{corners} that
\[ A = \{(0,0), (0,2), (2,0), (2,2), (2,3),
         (4,0), (4,3)\}
\]
is a $c_1$-freezing set of $X$. 
Notice also that $X$ can be differently described as 
$X=([0,4]_{\Z} \times [0,2]_{\Z}) \cup ([2,4]_{\Z} \times[0,3]_{\Z})$
from which Theorem~\ref{corners} yields a different freezing set,
\[ F = \{(0,0), (0,2), (2,0), (2,3), (4,0), (4,2), (4,3)\}.
\]
A minimal freezing set for $(X,c_1)$ that is a proper subset of $A$ is
\[ A' = \{(0,0), (4,0), (4,3), (2,3), (0,2) \}.
\]
\end{exl}

\begin{figure}
    \centering
    \includegraphics[height=2in]{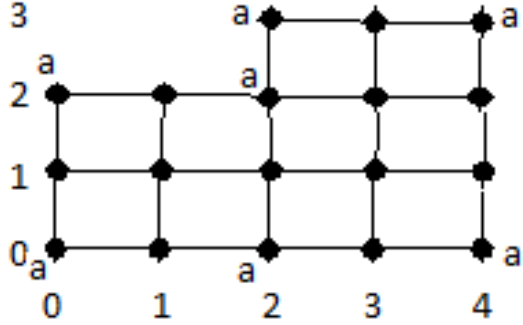}
    \caption{The digital image of
    Example~\ref{unionRectanglesExl}. Points
    of the freezing set $A$ are marked ``a".
    For the minimal freezing set $A' \subset A$, we have $\{(2,0), (2,2)\} \subset A \setminus A'$.
    }
    \label{fig:unionRectangles}
\end{figure}

\begin{proof}
First, we show $A'$ is a freezing set. Let $f \in C(X,c_1)$ be such that
$f|_{A'} = \id_{A'}$. From Proposition~\ref{uniqueShortestProp}, the line segments 
\begin{itemize}
    \item from $(0,0)$ to $(0,2)$,
    \item from $(0,0)$ to $(4,0)$,
    \item from $(4,0)$ to $(4,3)$, and
    \item from $(4,3)$ to $(2,3)$
\end{itemize}
all belong to $\Fix(f)$. Therefore, by Proposition~\ref{uniqueShortestProp}, the line segments 
\begin{itemize}
    \item from $(3,0)$ to $(3,3)$ and
    \item from $(2,0)$ to $(2,3)$
\end{itemize}
belong to $\Fix(f)$. Therefore, by Proposition~\ref{uniqueShortestProp}, the line segment
from $(0,2)$ to $(2,2)$ belongs to $\Fix(f)$.
Therefore, by Proposition~\ref{uniqueShortestProp}, the line segment
from $(1,0)$ to $(1,2)$ belongs to $\Fix(f)$. Thus $X = \Fix(f)$, so $A'$ is a freezing set
for $(X,c_1)$.

To show $A'$ is minimal, observe that for every $p \in A'$ there exists $q \in X$ such that
$q$ is a close $c_1$-neighbor of $p$:

$(1,1)$ is a close $c_1$-neighbor of both $(0,0)$ and $(0,2)$:

$(3,1)$ is a close $c_1$-neighbor of $(4,0)$; and

$(3,2)$ is a close $c_1$-neighbor of both $(2,3)$ and $(4,3)$.

It follows from Lemma~\ref{closeNbr} that $p \in A'$ implies
$A' \setminus \{p\}$ is not a freezing set for $(X,c_1)$. The assertion follows.
\end{proof}

In light of Theorem~\ref{convDiskThmActual},
perhaps Theorem~\ref{corners} will be especially
useful for $c_1$-connected images that are not
polygonal, as in the following.

\begin{exl}
\label{rectAndTrapzExl}
Let $X$ be the union of the horizontal segments
$[0,8]_{\Z} \times \{0\}$, $[0,3] \times \{1\}$,
$[0,3] \times \{2\}$,
$[6,8]_{\Z} \times \{1\}$, and 
$[7,8]_{\Z} \times \{2\}$ (see 
Figure~\ref{fig:rectAndTrapz}).
For the union $D_1 \cup D_2$ of thick convex disks that are subsets of $X$,
where
\[ D_1 = \{(x,y) \in X \, | \, x \le 3\}, ~~~  
   D_2 = \{x,y) \in X \, | \, x \ge 6\},
\] with $D_2$ considered
    with a bounding curve including the segment from $(7,2)$ to $(6,1)$
    (the dashed segment in Figure~\ref{fig:rectAndTrapz}),
Theorem~\ref{corners} gives for $(X,c_1)$ the
freezing set
\begin{equation}
\label{rAndTEeq}
    A=\left \{ \begin{array}{l}
     (0,0), (0,2), (3,0), (3,2), (4,0), (5,0), \\
     (6,0), (6,1), (7,2), (8,0), (8,2) 
\end{array} \right \}.
\end{equation}
A minimal freezing set $A' \subset A$ is
\[ A' = \{(0,0), (0,2), (3,2), (8,0), (8,2)\}.
\]
\end{exl}

\begin{figure}
    \centering
    \includegraphics[height=1in]{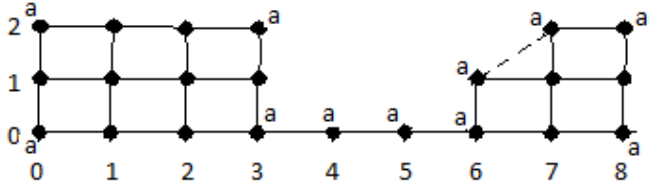}
    \caption{The digital image of
    Example~\ref{rectAndTrapzExl}. Points of the set
    $A$ of Theorem~\ref{corners} are marked ``a", where $A$ is based on
    the union $D_1 \cup D_2$ of thick convex disks that are subsets of $X$,
    where \newline $(x,y) \in D_1$ implies
    $x \le 3$, \newline $(x,y) \in D_2$
    implies $x \ge 6$, and \newline $D_2$ is
    considered with a bounding curve
    including the slanted segment from $(7,2)$
    to $(6,1)$.
        }
    \label{fig:rectAndTrapz}
\end{figure}

\begin{proof}
Let $f \in C(X,c_1)$ such that
$A' \subset \Fix(f)$. By~(\ref{rAndTEeq}) and
Proposition~\ref{uniqueShortestProp}, it follows
that the horizontal segments
$[0,8]_{\Z} \times \{0\}$ and
$[0,3]_{\Z} \times \{2\}$ belong to
$\Fix(f)$. It follows from
Proposition~\ref{uniqueShortestProp} that the
vertical segments $\{1\} \times [0,2]_{\Z}$
and $\{2\} \times [0,2]_{\Z}$ belong to
$\Fix(f)$. By Proposition~\ref{uniqueShortestProp}, the vertical
segment from $(8,0)$ to $(8,2)$ belongs to $\Fix(f)$. This much shows
$X \setminus \{(6,1), (7,1), (7,2)\} \subset \Fix(f)$.

Since $(6,1) \adj_{c_1} (6,0) \in \Fix(f)$, we must have
$p_1(f(6,1)) \in \{5, 6, 7\}$.
\begin{itemize}
    \item If $p_1(f(6,1)) = 5$ then by Lemma~\ref{cuPulling}, $p_1(f(7,1)) < 7$ and
          $p_1(f(8,1)) < 8$, a contradiction since $(8,1) \in \Fix(f)$.
    \item If $p_1(f(6,1)) = 7$ then the continuity of $f$ requires that 
           $(6,0) \not \in \Fix(f)$, a contradiction.
\end{itemize}
We conclude that $p_1(f(6,1)) = 6$.

Also since $(6,1) \adj_{c_1} (6,0) \in \Fix(f)$, we must have, by continuity of $f$,
$p_2(f(6,1)) \in \{0,1\}$.
If $p_2(f(6,1))=0$ then, since $f \in C(X,c_1)$, either $p_1(f(7,1)) =6$ or
$p_2(f(7,1))=0$. In either case, the continuity of $f$ would require
$(8,1) \not \in \Fix(f)$, a contradiction.
Therefore, we must have $p_2(f(6,1)) = 1$, so $(6,1) \in \Fix(f)$.

Therefore, $(7,1) \in \Fix(f)$, by  Proposition~\ref{uniqueShortestProp},
since $(7,1)$ is on the unique shortest path between the fixed points
$(6,1)$ and $(8,1)$.

Now we have $N(X,c_1,(7,2)) \subset \Fix(f)$, so the continuity of $f$ implies
that $(7,2) \in \Fix(f)$.

Thus $X = \Fix(f)$, so $A'$ is a freezing set.

To show $A'$ is minimal, note that every $p \in A'$ has a close
$c_1$-neighbor in $X$:
\[  \begin{array}{l}
    (1,1) \mbox{ is a close $c_1$-neighbor of both } (0,0) \mbox{ and } (0,2); \\
    (2,1) \mbox{ is a close $c_1$-neighbor of }  (3,2); \mbox{ and } \\
    (7,1) \mbox{ is a close $c_1$-neighbor of both }  (8,0) \mbox{ and } (8,2). 
\end{array} .
\]
From Lemma~\ref{closeNbr} it follows that $A'$ is a subset of every $c_1$-freezing
set of $X$. The assertion follows.
\end{proof}

\section{$c_2$ results}
In this section, we derive a result for the $c_2$ adjacency that is dual to
Theorem~\ref{corners}. We use the following.

\begin{thm}
{\rm \cite{BxConvex}}
\label{convDiskThmC2Actual}
Let $X$ be a thick convex disk with a  bounding
curve $S$. Let $B_1$ be the set of
points $x \in S$ such that
$x$ is an endpoint of a maximal
slanted edge in $S$. Let $B_2$ 
be the union of maximal horizontal 
and maximal vertical line segments in $S$.
Let $B = B_1 \cup B_2$. Then $B$ is a 
minimal freezing set for $(X,c_2)$
(see Figure~\ref{fig:convexRlts}(iii)).
\end{thm}

\begin{thm}
\label{slantCorners}
Let $V_i \subset X \subset \Z^2$, $i \in \{1,\ldots,n\}$ where
each $V_i$ is a thick convex disk. 
Let $X' = \bigcup_{i=1}^n V_i$.
Let $C_i$ be a bounding curve of $V_i$.
Let $B_{1,i}$ be the union of maximal horizontal and maximal vertical
segments of $C_i$.
Let $B_{2,i}$ be the set of endpoints of maximal slanted
segments of $C_i$. Then
$B = (X \setminus X') \cup \bigcup_{i=1}^n (B_{1,i} \cup B_{2,i})$ is a freezing set for $(X,c_1)$.
\end{thm}

\begin{proof}
Let $f \in C(X,c_2)$ such that $B \subset \Fix(f)$. By hypothesis $B_{1,i} \subset \Fix(f)$.
Let $S$ be a maximal
slanted segment of $C_i$. Since $B_{2,i} \subset \Fix(f)$, Proposition~\ref{uniqueShortestProp}
implies $S \subset \Fix(f)$.
It follows that $C_i \subset \Fix(f)$. Since $V_i$ is convex, for every $x \in V_i$
\begin{itemize}
    \item there is a vertical segment joining two members of
          $C_i$ and containing $x$; it follows from Lemma~\ref{cuPulling}
          that $p_1(f(x)) = p_1(x)$; and
    \item there is a horizontal segment joining two members of
          $C_i$ and containing $x$; it follows from Lemma~\ref{cuPulling}
          that $p_2(f(x)) = p_2(x)$. Hence $x \in \Fix(f)$.
\end{itemize}
Thus, for all $i$, $V_i \subset \Fix(f)$.
Since by hypothesis, $X \setminus X' \subset \Fix(f)$, it follows that 
$X = \Fix(f)$. Since $f$ is arbitrary, the assertion follows.
\end{proof}

\begin{figure}
    \centering
    \includegraphics[height=1.5in]{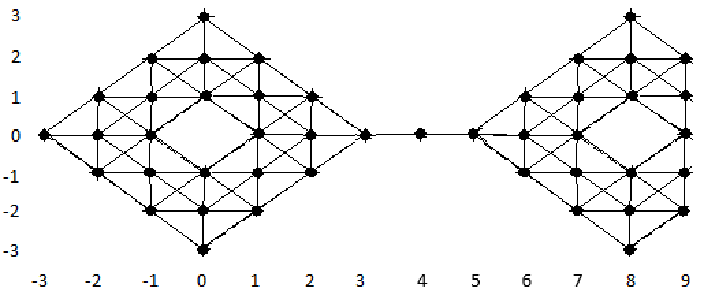}
    \caption{The digital image of Example~\ref{maskExl}. $X= \{(4,0)\} \cup \bigcup_{i=1}^8 D_i$,
    where 
    the $D_i$ are the thick convex disks listed below.\newline
 Subsets of $\{(x,y) \in X \, | \, x \le 3\}$: \newline
 $~~~~~D_1$, with hull vertices $\{(-3,0), (0,3), (1,2), (-2,-1) \}$; \newline
 $~~~~~D_2$, with hull vertices  $\{(1,2), (3,0), (2,-1), (0,1) \}$; \newline
 $~~~~~D_3$, with hull vertices  $\{(2,-1), (0,-3), (-1,-2), (1,0) \}$; and \newline
 $~~~~~D_4$, with hull vertices $\{(-1,-2), (-2,-1), (-1,0), (0,-1) \}$. \newline
 Subsets of $\{(x,y) \in X \, | \, x \ge 5\}$: \newline
     $~~~~~D_5$, with hull vertices $\{(5,0), (7,2), (8,1), (6,-1) \}$, \newline
    $~~~~~D_6$, with hull vertices $\{(7,2), (8,3), (9,2), (9,0) \}$, \newline
    $~~~~~D_7$, with hull vertices  $\{(9,0), (9,-2), (8,-3), (7,-2)\}$, and \newline
      $~~~~~D_8$, with hull vertices $\{(7,-2), (6,-1), (7,0), (8,-1) \}$.
   } 
    \label{fig:mask}
\end{figure}

\begin{exl}
\label{maskExl}
Let $X \subset \Z^2$ be the digital image shown in Figure~\ref{fig:mask}. 
The hull vertices listed for disks $D_i$ in this figure 
are all endpoints of maximal slanted bounding edges or members of horizontal or vertical
bounding edges of their
respective $D_i$. By Theorem~\ref{slantCorners}, these hull vertices of the $D_i$;
$(9,1)$ and $(9,-1)$, members of vertical
bounding edges of $D_6$ and $D_7$, respectively; and 
$(4,0) \in X \setminus \bigcup_{i=1}^8 D_i$, make up
a freezing set $B$ for $(X,c_2)$. Thus a listing of 
members of $B$ (note there are vertices that belong to more than one $D_i$):

$B =$
\[\left \{ \begin{array}{l}
     (-3,0), (0,3), (1,2), (-2,-1), (3,0), (2,-1), (0,1), (0,-3), (-1,-2),  \\
     (1,0), (-1,0), (0,-1), (3,0), (4,0), (5,0), (7,2), (8,1), (6,-1), (8,3), (9,2), \\
     (9,1),(9,0), (9,-1), (9,-2), (8,-3), (7,-2), (6,-1), (7,0), (8,-1)
\end{array}
    \right \}
\]

Let $B' \subset B$ be the set
\[ B' = \{(-3,0), (0,-3), (0,3), (8,-3), (8,3), (9,-2), (9,-1), (9,1), (9,2)\}.
\]
Then $B'$ is a minimal freezing set for $(X,c_2)$.
\end{exl}

\begin{proof}
Let $f \in C(X,c_2)$ such that $B' \subset \Fix(f)$. By
Proposition~\ref{uniqueShortestProp}, we have the following.
\begin{itemize}
    \item The line segment $S_1$ from $(-3,0)$ to $(0,-3)$ belongs to $\Fix(f)$.
    \item The line segment $S_2$ from $(-3,0)$ to $(0,3)$  belongs to $\Fix(f)$.
    \item The path $S_3$ consisting of the line segment from $(0,-3)$ to $(3,0)$, the
          line segment from $(3,0)$ to $(5,0)$, and the line segment from
          $(5,0)$ to $(8,-3)$, belongs to $\Fix(f)$.
   \item The path $S_4$ consisting of the line segment from $(0,3)$ to $(3,0)$, the
          line segment from $(3,0)$ to $(5,0)$, and the line segment from
          $(5,0)$ to $(8,3)$, belongs to $\Fix(f)$.
    \item The line segment $S_5$ from $(8,-3)$ to $(9,-2)$ belongs to $\Fix(f)$.
    \item The line segment $S_6$ from $(8,3)$ to $(9,2)$ belongs to $\Fix(f)$.
\end{itemize}
Also, by hypothesis, the line segment $S_7$ from $(9,-1)$ to $(9,1)$ 
belongs to $\Fix(f)$. By the convexity of the $V_i$, every 
$x \in B' \setminus \bigcup_{k=1}^7 S_k$ belongs to a horizontal 
line segment between two members of $\bigcup_{k=1}^7 S_k$; hence
by Lemma~\ref{cuPulling}, $p_1(f(x)) = p_1(x)$. Also by the convexity of the $V_i$, every 
$x \in B' \setminus \bigcup_{k=1}^7 S_k$ belongs to a vertical 
line segment between two members of $\bigcup_{k=1}^7 S_k$; hence
by Lemma~\ref{cuPulling}, $p_2(f(x)) = p_2(x)$. Thus $x \in \Fix(f)$. 
Thus $X = \Fix(f)$, so $B'$ is a freezing set.

Notice that every $p \in B'$ has a
close $c_2$-neighbor in~$X$, as listed below.
\[ \begin{array}{llc}
   p \in B' & ~~~~~ & \mbox{close } c_2 \mbox{ neighbor of } p \mbox{ in } (X,c_2)\\
   (-3,0) & ~~~~~ & (-2,0) \\
   (0,-3) & ~~~~~ & (0,-2) \\
   (0,3) & ~~~~~ & (0,2) \\
   (8,-3) & ~~~~~ & (8,-2) \\
   (8,3) & ~~~~~ & (8,2) \\
   (9,-2) & ~~~~~ & (8,-2) \\
   (9,-1) & ~~~~~ & (8,-1) \\
   (9,1) & ~~~~~ & (8,1) \\
   (9,2) & ~~~~~ & (8,2)
\end{array}
\]
By Lemma~\ref{closeNbr}, $p$ belongs to every
freezing set of $(X,c_2)$. Therefore, $B'$ is minimal.
\end{proof}

\section{Further remarks}
Theorems~\ref{corners} and~\ref{slantCorners} give methods for finding a freezing set for
$(X,c_1) \subset \Z^2$ or $(X,c_2) \subset \Z^2$, respectively. Roughly, a freezing set
is found by filling $X$ as much as possible by thick convex disk subsets,
then using the formula of the respective theorem.
For both $c_1$ and $c_2$, the resulting freezing set can be examined, often using
tools used in our examples, for a subset
that is a minimal freezing set.


\begin{thebibliography}{99}



\bibitem{Bx99}
 L. Boxer, A classical construction for the digital fundamental group,
 {\em Journal of Mathematical Imaging and Vision} 10 (1999), 51-62.

 
https://link.springer.com/article/10.1023/A





\bibitem{BxFpSets}
L. Boxer, Fixed point sets in digital topology, 2, 
{\em Applied General Topology} 21(1) (2020), 111-133.

https://polipapers.upv.es/index.php/AGT/article/view/12101

\bibitem{BxConvex}
L. Boxer, Convexity and Freezing Sets in Digital Topology, submitted. Available at
https://arxiv.org/abs/2005.09713

 
 
 
 


\bibitem{bs19a}
L. Boxer and P.C. Staecker,
Fixed point sets in digital topology, 1,
{\em Applied General Topology} 21 (1) (2020), 87-110.

https://polipapers.upv.es/index.php/AGT/article/view/12091


\bibitem{Chen94}
L. Chen,
{\em Gradually varied surface and its optimal uniform approximation},
SPIE Proceedings 2182 (1994), 300-307.

https://www.spiedigitallibrary.org/conference-proceedings-of-spie/2182/0000/Gradually-varied-surface-and-its-optimal-uniform-approximation/10.1117/12.171078.short

\bibitem{Chen04}
L. Chen, {\em Discrete Surfaces and Manifolds},
Scientific Practical Computing, Rockville, MD, 2004.

https://www.amazon.com/Discrete-Surfaces-Manifolds-Digital-Discrete-Geometry/dp/0975512218


\bibitem{hmps} 
J. Haarmann, M.P. Murphy, C.S. Peters, and P.C. Staecker,
Homotopy equivalence in finite digital images,
{\em Journal of Mathematical Imaging and Vision}
53 (2015), 288-302.

https://link.springer.com/article/10.1007/s10851-015-0578-8




\bibitem{RosenfeldMAA}
A. Rosenfeld,
Digital topology,
{\em The American Mathematical Monthly} 86 (8) (1979), 621-630.

https://www.jstor.org/stable/2321290?seq=1

\bibitem{Rosenfeld}
A. Rosenfeld, `Continuous' functions on digital pictures, {\em Pattern Recognition Letters} 4,
pp. 177-184, 1986.

https://www.sciencedirect.com/science/article/pii/0167865586900176

\end{thebibliography}
\end{document}